 \documentclass[final,3p,times]{elsarticle}

\usepackage{amssymb}
\usepackage{graphics,graphicx,amsmath,latexsym,amssymb,amsthm,amssymb,amsfonts}
\newtheorem{thm}{Theorem}[section]

\newtheorem{defn}[thm]{Definition}
\newtheorem{rem}[thm]{Remark}

\newtheorem{exm}[thm]{Example}

\biboptions{square}

\journal{x}

\begin{document}

\begin{frontmatter}


\author{J. Mahanta 
}
 \ead{jm$\_$nerist@yahoo.in}
 \author{P. K. Das 
}
\ead{pkd$\_$ma@yahoo.cam}
\address{Department of Mathematics\\
NERIST, Nirjuli\\
Arunachal Pradesh, 791 109, INDIA.\\ }

\title{On soft topological space via semiopen and semiclosed soft sets}



\begin{abstract}
This paper introduces semiopen and semiclosed soft sets in soft topological spaces. The notions of interior and closure are generalized using these sets. A detail study is carried out on properties of semiopen, semiclosed soft sets, semi interior and semi closure of a soft set in a soft topological space. Various forms of soft functions, like semicontinuous, irresolute, semiopen soft functions are introduced and characterized. Further soft semicompactness, soft semiconnectedness and soft semiseparation axioms are introduced and studied.
\end{abstract}


\end{frontmatter}
\section{Introduction and Preliminaries} 
Soft set was introduced by Molodtsov \cite{S1} in the year 1999. The notion of topological space for soft sets was formulated by Shabir et. al. \cite{S2}. Of late many authors have studied various properties of soft topological spaces.

This paper aims to introduce and give a detail study of semiopen soft set, semiclosed soft set, semicontunuity, semicompactness, semiconnectedness and semiseparation axioms.

Here are some definitions and results required in the sequel.
\begin{defn} \cite{S2}
 Let $\tau$ be a collection of soft sets over a universe $U$ with a fixed set $A$ of parameters, then $\tau \subseteq SS(U)_A$ 
is called a soft topology on $U$ with a fixed set $A$ if 
\begin{enumerate}[(i)]
\item $\Phi_A, U_A$ belongs to $\tau$;
\item the union of any number of soft sets in $\tau $ belongs to $\tau$;
\item the intersection of any two soft sets in $\tau $ belongs to $\tau$.
\end{enumerate}
Then $(U_A, \tau)$ is called a soft topological space over $U$.
\end{defn}

\begin{defn} \cite{S4}
A soft basis of a soft topological space $(U_A, \tau)$ is a subcollection $\mathcal{B}$ of open soft sets such that every element of $\tau$ can be expressed as the union of elements of $\mathcal{B}$.  

\end{defn}

\begin{defn} \cite{S4}
Let $(U_A, \tau)$ be a soft topological space and $U_B \overset{\sim }{\subseteq} U_A$. then the collection
${\tau}_{U_B}= \{ U_{A_i} \overset{\sim }{\bigcap} U_B~|~ U_{A_i} \in \tau, i\in I \subseteq \mathbb{N}\}$ is called a soft subspace topology on $U_B$.
\end{defn}

\begin{defn} \cite{S3}
A soft set $F_A \in SS(U)_A$ is called a soft point in $U_A$, denoted by $e_F$, if for each element $e \in A, F(e) \neq \Phi$ and $F(e^{'}) = \Phi$ for all $e^{'} \in A-\{e\}$.
\end{defn}

\begin{defn} \cite{S3}
A soft point $e_F$ in said to be in the soft set $G_A$, denoted by $e_F \overset{\sim}{\in} G_A$, if for some element $e \in A$ and $F(e)=G(e)$.
\end{defn}

\begin{defn} \cite{S3}
A family $\Psi $ of soft sets is a cover of a soft set $F_A$ if
 $F_ A \overset{\sim }{\subseteq} \overset{\sim }{\bigcup} \{(F_i)_ A ~|~ (F_i)_ A \in \Psi, i \in I\}$.
 
A subcover of $\Psi $ is a subfamily of $\Psi $ which is also a cover.  
\end{defn}
\begin{defn} \cite{S3}
A family $\Psi $ of soft sets has the finite intersection property (FIP) if the intersection of the members of each finite subfamily of $\Psi$ is not null soft set.
\end{defn}

Throughout this study, $F_A$ denotes a soft set, $(U_A, \tau)$ denotes a soft topological space.

\section{Semiopen and semiclosed soft sets}

In this section, we introduce semiopen and semiclosed soft sets and study various notions related to this structure.
\begin{defn}
In a soft topological space $(U_A, \tau)$, a soft set 
\begin{enumerate}[(i)]
 \item $G_ A$ is said to be semiopen soft set if $\exists$ an open soft set $H_A$ such that\\ $H_A \overset{\sim }{\subseteq} G_A \overset{\sim }{\subseteq} \overline{H_A}$;
\item $F_ A$ is said to be semiclosed soft set if $\exists$ a closed soft set $K_A$ such that\\ $K_A^0 \overset{\sim }{\subseteq} F_A \overset{\sim }{\subseteq} K_A$;
\end{enumerate}

\end{defn}

\begin{exm}
Consider the soft topological spaces $(U_A, \tau)$ as defined in Example 3 of \cite{S2}. Here $G_A(e_1)=\{h_1, h_2\}, G_A(e_2)=\{h_1\}$ is a semiopen soft set, as $F_{1A}$ is a open soft set such that $F_{1A} \overset{\sim }{\subseteq} G_A \overset{\sim }{\subseteq} \overline{F_{1A}}=F_{1A}$. 

$K_A(e_1)=\{h_3\}, K_A(e_2)=\{h_3\}$ is a semiclosed soft set, as $(F_{1A})^c$ is a closed soft set such that $((F_{1A})^c)^0 \overset{\sim }{\subseteq} K_A \overset{\sim }{\subseteq} (F_{1A})^c$.
\end{exm}

\begin{rem}
Every open (closed) soft set is a semiopen (semiclosed) soft set but not conversely. 
\end{rem}
\begin{rem}
$\Phi_A$ and $U_A$ are always semiclosed and semiopen.
\end{rem}

From now onwards, we shall denote the family of all semiopen soft sets (semiclosed soft sets) of a soft topological space $(U_A, \tau)$ by $SOSS(U)_A $ ($SCSS(U)_A)$.

\begin{thm}
Arbitrary union of semiopen soft sets is a semiopen soft set.
\end{thm}
\begin{proof}
Let $\{(G_A)_\lambda~|~ \lambda \in \Lambda\}$ be a collection of semiopen soft sets of a soft topological space $(U_A, \tau)$. Then $\exists$ an open soft sets $(H_A)_\lambda$ such that $(H_A)_\lambda \overset{\sim }{\subseteq} (G_A)_\lambda \overset{\sim }{\subseteq} \overline{(H_A)_\lambda}$ for each $\lambda$; hence $\overset{\sim }{\bigcup}(H_A)_\lambda \overset{\sim }{\subseteq} \overset{\sim }{\bigcup}(G_A)_\lambda \overset{\sim }{\subseteq}  \overline{\overset{\sim }{\bigcup}(H_A)_\lambda}$ and $\overset{\sim }{\bigcup}(H_A)_\lambda$ is open soft set. So, it is concluded that $\overset{\sim }{\bigcup}(G_A)_\lambda$ is a semiopen soft set.
\end{proof}

\begin{rem}
Arbitrary intersection of semiclosed soft sets is a semiclosed soft set.
\end{rem}
\begin{thm}
If a semiopen soft set $G_A$ is such that $G_A \overset{\sim }{\subseteq} K_A \overset{\sim }{\subseteq} \overline {G_A} $, then $K_A$ is also semiopen.
\end{thm}
\begin{proof}
As $G_A$ is semiopen soft set $\exists$ an open soft set $H_A$ such that $H_A \overset{\sim }{\subseteq} G_A \overset{\sim }{\subseteq} \overline{H_A}$; then by hypothesis $H_A \overset{\sim }{\subseteq}  K_A$ and $\overline{G_A} \overset{\sim }{\subseteq} \overline{H_A} \Rightarrow K_A \overset{\sim }{\subseteq} \overline{G_A} \overset{\sim }{\subseteq} \overline{H_A} $ i.e., $H_A \overset{\sim }{\subseteq} K_A \overset{\sim }{\subseteq} \overline{H_A}$, hence $K_A$ is a semiopen soft set.
\end{proof}
\begin{thm}
If a semiclosed soft set $F_A$ is such that $(F_A)^0 \overset{\sim }{\subseteq} K_A \overset{\sim }{\subseteq} F_A $, then $K_A$ is also semiclosed.
\end{thm}
\begin{thm}
A soft set $G_A \in SOSS(U)_A \Leftrightarrow$ for every soft point $e_G \overset{\sim}{\in} G_A, \exists$ a soft set $H_A \in SOSS(U)_A$ such that $e_G \overset{\sim}{\in} H_A \overset{\sim }{\subseteq}G_A$.
\end{thm}
\begin{proof}

$(\Rightarrow)$ Take $H_A=G_A$.

$(\Leftarrow)$ $G_A= \overset{\sim}{\underset{e_G \overset{\sim}{\in} G_A}{\bigcup}}(e_G) \overset{\sim}{\subseteq} \overset{\sim}{\underset{e_G \overset{\sim}{\in} G_A}{\bigcup}}H_A \overset{\sim}{\subseteq} G_A$.
\end{proof}

\begin{defn}
Let $(U_A, \tau)$ be a soft topological space and $G_A$ be a soft set over $U$.
\begin{enumerate} [(i)]
\item The soft semi closure of $G_A$ is a soft set\\ 
$ssclG_A= \overset{\sim }{\bigcap} \{S_A~|~ G_A \overset{\sim }{\subseteq} S_A$ and $ S_A \in SCSS(U)_A \}$;
\item The soft semi interior of $G_A$ is a soft set\\ 
$ssintG_A= \overset{\sim }{\bigcup} \{S_A~|~ S_A \overset{\sim }{\subseteq} G_A$ and $ S_A \in SOSS(U)_A \}$.
\end{enumerate}
\end{defn}

$ssclG_A$ is the smallest semiclosed soft set containing $G_A$ and $ssintG_A$ is the largest semiopen set contained in $G_A$.

\begin{thm}
Let $(U_A, \tau)$ be a soft topological space and $G_A$ and $K_A$ be two soft sets over $U$, then 
\begin{enumerate}[(i)]
\item $ G_A \in SCSS(U)_A  \Leftrightarrow G_A= ssclG_A$;
\item $ G_A \in SOSS(U)_A  \Leftrightarrow G_A= ssintG_A$;
\item $(ssclG_A)^c = ssint(G_A^c)$;
\item $(ssintG_A)^c = sscl(G_A^c)$;
\item $G_A \overset{\sim }{\subseteq} K_A \Rightarrow ssintG_A \overset{\sim }{\subseteq} ssintK_A$;
\item $G_A \overset{\sim }{\subseteq} K_A \Rightarrow ssclG_A \overset{\sim }{\subseteq} ssclK_A$;
\item $sscl \Phi _A = \Phi _A$ and $sscl U _A = U _A$;
\item $ssint \Phi _A = \Phi _A$ and $ssint U _A = U _A$;
\item $sscl(G_A \overset{\sim }{\cup} K_A) = ssclG_A \overset{\sim }{\cup}  ssclK_A$;
\item $ssint(G_A \overset{\sim }{\cap} K_A) = ssintG_A \overset{\sim }{\cap}  ssintK_A$;
\item $sscl(G_A \overset{\sim }{\cap} K_A) \overset{\sim }{\subset} ssclG_A \overset{\sim }{\cap}  ssclK_A$;
\item $ssint(G_A \overset{\sim }{\cup} K_A) \overset{\sim }{\subset} ssintG_A \overset{\sim }{\cup}  ssintK_A$;
\item $sscl(ssclG_A)=ssclG_A$;
\item $ssint(ssintG_A)=ssintG_A$.
\end{enumerate}
\end{thm}

\begin{proof} Let $G_A$ and $K_A$ be two soft sets over $U$.
\begin{enumerate}[(i)]
\item Let $G_A$ be a semiclosed soft set. Then it is the smallest semiclosed set containing itself and hence $G_A= ssclG_A$. 

On the other hand, let $G_A= ssclG_A$ and $ssclG_A \in SCSS(U) \Rightarrow G_A \in SCSS(U)$.
\item Similar to $(i)$.
\item \begin{eqnarray*}
(ssclG_A)^c & = & (\overset{\sim }{\bigcap} \{S_A~|~ G_A \overset{\sim }{\subseteq} S_A and  S_A \in SCSS(U)_A \})^c \\
&=& \overset{\sim }{\bigcup} \{S_A^c~|~ G_A \overset{\sim }{\subseteq} S_A and  S_A \in SCSS(U)_A \} \\
&= &\overset{\sim }{\bigcup} \{S_A^c~|~ S_A^c \overset{\sim }{\subseteq} G_A^c and  S_A^c \in SOSS(U)_A \} \\ 
&=& ssint(G_A^c).
\end{eqnarray*}
\item Similar to $(iii)$.
\item Follows from definiton.
\item Follows from definition.
\item Since $\Phi _A$ and $U_A$ are semiclosed soft sets so $sscl \Phi _A = \Phi _A$ and $sscl U _A = U _A$.  
\item Since $\Phi _A$ and $U_A$ are semiopen soft sets so $ssint \Phi _A = \Phi _A$ and $ssint U _A = U _A$. 
\item We have $G_A \overset{\sim }{\subset} G_A \overset{\sim }{\bigcup} K_A$ and  $K_A \overset{\sim }{\subset} G_A \overset{\sim }{\bigcup} K_A$. Then by $(vi), ssclG_A \overset{\sim }{\subset} sscl(G_A \overset{\sim }{\bigcup} K_A)$ and $ssclK_A \overset{\sim }{\subset} sscl(G_A \overset{\sim }{\bigcup} K_A) \Rightarrow ssclK_A\overset{\sim }{\bigcup} ssclG_A \overset{\sim }{\subset} sscl(G_A \overset{\sim }{\bigcup} K_A)$. 

Now, $ssclG_A, ssclK_A \in SCSS(U)_A \Rightarrow ssclG_A \overset{\sim }{\bigcup} ssclK_A \in SCSS(U)_A$.

Then $G_A \overset{\sim }{\subset} ssclG_A$ and $K_A \overset{\sim }{\subset} ssclK_A$ imply $G_A\overset{\sim }{\bigcup}K_A \overset{\sim }{\subset} ssclG_A \overset{\sim }{\bigcup}ssclK_A$.i.e., $ssclG_A \overset{\sim }{\bigcup}ssclK_A$ is a semiclosed set containing $G_A\overset{\sim }{\bigcup}K_A$. But $sscl(G_A \overset{\sim }{\bigcup}K_A)$ is the smallest semiclosed soft set containing $G_A\overset{\sim }{\bigcup}K_A$. Hence $sscl(G_A \overset{\sim }{\bigcup}K_A) \overset{\sim }{\subset} ssclG_A \overset{\sim }{\bigcup}ssclK_A$. 
 So, $sscl(G_A \overset{\sim }{\cup} K_A) = ssclG_A \overset{\sim }{\cup}  ssclK_A$.
\item Similar to $(ix)$.
\item We have $G_A \overset{\sim }{\bigcap} K_A \overset{\sim }{\subset} G_A$ and $G_A \overset{\sim }{\bigcap} K_A \overset{\sim }{\subset} K_A\\ \Rightarrow sscl(G_A \overset{\sim }{\bigcap} K_A) \overset{\sim }{\subset} ssclG_A$ and $sscl(G_A \overset{\sim }{\bigcap} K_A) \overset{\sim }{\subset} ssclK_A \\ \Rightarrow sscl(G_A \overset{\sim }{\bigcap} K_A) \overset{\sim }{\subset} ssclG_A\overset{\sim }{\bigcap} ssclK_A$. 
\item Similar to $(xi)$.
\item Since $ssclG_A \in SCSS(U)$ so by $(i), sscl(ssclG_A)=ssclG_A$.
\item Since $ssintG_A \in SOSS(U)$ so by $(ii), ssint(ssintG_A)=ssintG_A$.
\end{enumerate}
\end{proof}

\begin{thm}
If $G_A$ is any soft set in a soft topological space $(U_A, \tau)$ then following are equivalent: 
\begin{enumerate} [(i)]
\item $G_A$ is semiclosed soft set;
\item $(\overline{G_A})^0 \overset{\sim }{\subseteq} G_A$;
\item $\overline{(G_A^c)^0} \overset{\sim }{\supseteq} G_A^c$.
\item $G_A^c$ is semiopen soft set;

\end{enumerate}
\end{thm}

\begin{proof}

$(i) \Rightarrow (ii)$ If $G_A$ is semiclosed soft set, then $\exists$ closed soft set $H_A$ such that $H_A^0 \overset{\sim }{\subseteq} G_A \overset{\sim }{\subseteq} H_A \Rightarrow H_A^0 \overset{\sim }{\subseteq} G_A \overset{\sim }{\subseteq} \overline{G_A} \overset{\sim }{\subseteq}H_A$. By the property of interior we then have $(\overline{G_A})^0 \overset{\sim }{\subseteq} H_A^0 \overset{\sim }{\subseteq} G_A$; 

$(ii) \Rightarrow (iii) (\overline{G_A})^0 \overset{\sim }{\subseteq} G_A \Rightarrow G_A^c \overset{\sim }{\subseteq} ((\overline {G_A})^0)^c = \overline{(G_A^c)^0} \overset{\sim}{\supseteq} G_A^c$.

$(iii) \Rightarrow (iv)$ $H_A = (G_A^c)^0$ is an open soft set such that $(G_A^c)^0 \overset{\sim }{\subseteq} G_A^c 
\overset{\sim }{\subseteq} \overline{(G_A^c)^0} $, hence $G_A^c$ is semiopen.

$(iv) \Rightarrow (i)$ As $G_A^c$ is semiopen $\exists$ an open soft set $H_A$ such that $H_A \overset{\sim }{\subseteq} G_A^c \overset{\sim }{\subseteq} \overline{H_A} \Rightarrow H_A^c$ is a closed soft set such that $G_A \overset{\sim }{\subseteq} H_A^c$ and  $G_A^c \overset{\sim }{\subseteq} \overline{H_A} \Rightarrow (H_A^c)^0 \overset{\sim }{\subseteq} G_A$, hence $G_A$ is semiclosed soft set. 
\end{proof}

\section{Soft Semicontinuous,  Soft Irresolute, Soft Semiopen and Soft Semoclosed Functions }
Here we introduce different types of soft functions in soft topological spaces and investigate their properties.
\begin{defn}
Let $(U_A, \tau)$ and $(U_B, \delta)$ be two soft topological spaces. A soft function $f: U_A \rightarrow U_B$ is
 said to be 
\begin{enumerate}[(i)]
\item soft semicontinuous if for each soft open set $G_B$ of $U_B$, the inverse image $f^{-1}(G_B)$ is soft semiopen set of $U_A$;
\item soft irresolute if for each soft semiopen set $G_B$ of $U_B$, the inverse image $f^{-1}(G_B)$ is soft semiopen set of $U_A$;
\item soft semiopen function if for each open soft set $G_A$ of $U_A$, the image $f(G_A)$ is semiopen soft set of $U_B$;
\item soft semiclosed function if for each closed soft set $F_A$ of $U_A$, the image $f(F_A)$ is semiclosed soft set of $U_B$.
\end{enumerate}
\end{defn}

\begin{rem}
\begin{enumerate}[(a)]
\item A soft function $f: U_A \rightarrow U_B$ is soft semicontinuous if for each soft closed set $F_B$ of $U_B$, the inverse image $f^{-1}(F_B)$ is soft semiclosed set of $U_A$.
\item A soft semicontinuous function is soft irresolute.
\end{enumerate}
\end{rem}

\begin{thm}
A soft function $f: U_A \rightarrow U_B$ is soft semicontinuous iff $f(ssclF_A) \overset{\sim}{\subseteq} \overline{f(F_A)}$ for every soft set $F_A$ of $U_A$.
\end{thm}
\begin{proof}
Let $f: U_A \rightarrow U_B$ is soft semicontinuous. Now $\overline{f(F_A)}$ is a soft closed set of $U_B$, so by soft semicontinuity of $f$, $f^{-1}(\overline{f(F_A)})$ is soft semiclosed and 
$F_A \overset{\sim}{\subseteq} f^{-1}(\overline{f(F_A)})$. But $ssclF_A$ is the smallest semiclosed set containing $F_A$, hence $ssclF_A \overset{\sim}{\subseteq} f^{-1}(\overline{f(F_A)}) \Rightarrow 
f(ssclF_A) \overset{\sim}{\subseteq} \overline{f(F_A)}$.

Conversely, let $F_B$ be any soft closed set of $U_B \Rightarrow f^{-1}(F_B) \in U_A \Rightarrow 
f(sscl(f^{-1}(F_B))) \overset{\sim}{\subseteq} \overline{f(f^{-1}(F_B))} \Rightarrow f(sscl(f^{-1}(F_B))) \overset{\sim}{\subseteq} \overline{F_B} = F_B \Rightarrow sscl(f^{-1}(F_B)) = f^{-1}(F_B)$, hence is semiclosed.  
\end{proof}

\begin{thm}
A soft function $f: U_A \rightarrow U_B$ is soft semicontinuous iff $f^{-1}(H_B)^0 \overset{\sim}{\subseteq} ssint(f^{-1}(H_B))$ for every soft set $H_B$ of $U_B$.
\end{thm}

\begin{proof}
Let $f: U_A \rightarrow U_B$ is soft semicontinuous. Now $(f(G_A))^0$ is a soft open set of $U_B$, so by soft semicontinuity of $f$, $f^{-1}(f(G_A))^0$ is soft semiopen and $f^{-1}(f(G_A))^0 \overset{\sim}{\subseteq} G_A$. As $ssintG_A$ is the largest soft semiopen set contained in $G_A$, $f^{-1}(f(G_A))^0 \overset{\sim}{\subseteq} ssintG_A$.

Conversely, take a soft open set $G_B \Rightarrow f^{-1}(G_B)^0 \overset{\sim}{\subseteq} ssint(f^{-1}(G_B)) \Rightarrow f^{-1}(G_B) \overset{\sim}{\subseteq} ssint(f^{-1}(G_B)) \Rightarrow f^{-1}(G_B)$ is soft semiopen.

\end{proof}

\begin{thm}
A soft function $f: U_A \rightarrow U_B$ is soft semiopen iff $f((F_A)^0) \overset{\sim}{\subseteq} ssint(f(F_A))$ for every soft set $F_A$ of $U_A$.
\end{thm}
\begin{proof}
If $f: U_A \rightarrow U_B$ is soft semiopen, then $f((F_A)^0)= ssintf((F_A)^0) \overset{\sim}{\subseteq} ssintf(F_A)$.

On the other hand, take a soft open set $G_A$ of $U_A$. Then by hypothesis, $f(G_A) = f((G_A)^0) \overset{\sim}{\subseteq} ssint(f(G_A)) \Rightarrow f(G_A)$ is soft semiopen in $U_B$. 
\end{proof}

\begin{thm}
Let $f: U_A \rightarrow U_B$ be soft semiopen. If $K_B$ is a soft set and $F_A$ is closed soft set containing $f^{-1}(K_B)$ then $\exists$ a semiclosed soft set $H_B$ such that $K_B \overset{\sim}{\subseteq} H_B$ and $f^{-1}(H_B) \overset{\sim}{\subseteq} F_A$.
\end{thm}
\begin{proof}
Take $H_B = (f(F_A^c))^c$. Now $f^{-1}(K_B) \overset{\sim}{\subseteq} F_A \Rightarrow f(F_A^c) \overset{\sim}{\subseteq} K_B^c$. Then $F_A^c open \Rightarrow f(F_A^c)$ is semiopen, so $H_B$ is semiclosed and $K_B \overset{\sim}{\subseteq} H_B$ and $f^{-1}(H_B) \overset{\sim}{\subseteq} F_A$. 
\end{proof}

\begin{thm}
A soft function $f: U_A \rightarrow U_B$ is soft semiclosed iff $ssclf(F_A) \overset{\sim}{\subseteq} f(\overline{F_A})$ for every soft set $F_A$ of $U_A$.
\end{thm}

\section{Semicompact soft topological spaces}

This section is devoted to introduce semicompactness in soft topological spaces along with characterization of semicompact soft topological spaces.

\begin{defn}
A cover of a soft set is said to be a semiopen soft cover if every member of the cover is a semiopen soft set.
\end{defn}

\begin{defn}
A soft toplogical space $(U_A, \tau)$ is said to be semicompact if each semiopen soft cover of $U_A$ has a finite subcover.
\end{defn}

\begin{rem}
Every compact soft space is also semicompact.
\end{rem}
\begin{thm}
A soft topological space $(U_A, \tau)$ is semicompact $\Leftrightarrow$ each family of semiclosed soft sets with the FIP has a nonempty intersection.
\end{thm}

\begin{proof}
Let $\{(F_A)_\lambda~|~\lambda \in \Lambda\}$ be a collection of semiclosed soft sets with the FIP. If possible, assume 
$\underset{\lambda \in \Lambda }{\overset{\sim}{\bigcap}} (F_A)_\lambda = \Phi_A 
\Rightarrow \underset{\lambda \in \Lambda }{\overset{\sim}{\bigcup}} ((F_A)_\lambda)^c = U_A$.
So, the collection $\{((F_A)_\lambda)^c~|~ \lambda \in \Lambda\}$ forms a soft semiopen cover of $U_A$, which is semicompact. So, 
$\exists$ a finite subcllection $\Delta$ of $\Lambda$ which also covers $U_A$. i.e., $\underset{\lambda \in \Delta }{\overset{\sim}{\bigcup}} ((F_A)_\lambda)^c = U_A \Rightarrow
\underset{\lambda \in \Delta }{\overset{\sim}{\bigcap}} (F_A)_\lambda = \Phi_A$, a contradiction.

For the converse, if possible, let $(U_A, \tau)$ be not semicompact. Then $\exists$ a semiopen cover 
$\{(G_A)_\lambda ~|~ \lambda \in \Lambda\}$ of $U_A$, such that for every finite subcollection $\Delta$ of $\Lambda$ we have 
$\underset{\lambda \in \Delta }{\overset{\sim}{\bigcup}} (G_A)_\lambda \neq U_A \Rightarrow \underset{\lambda \in \Delta }{\overset{\sim}{\bigcap}} ((G_A)_\lambda)^c \neq \Phi_A $.
 Hence $\{((G_A)_\lambda)^c~|~ \lambda \in \Lambda\}$ has the FIP. So, by hypothesis $\underset{\lambda \in \Lambda}{\overset{\sim}{\bigcap}} ((G_A)_\lambda)^c \neq \Phi_A \Rightarrow
\underset{\lambda \in \Lambda }{\overset{\sim}{\bigcup}} (G_A)_\lambda \neq U_A $, a contradiction.
\end{proof}
 
\begin{thm}
A soft topological space $(U_A, \tau)$ is semicompact iff every family $\Psi$ of soft sets with the FIP, $\underset{G_A \in \Psi }{\overset{\sim}{\bigcap}} ssclG_A \neq \Phi_A$.
\end{thm}
\begin{proof}
Let $(U_A, \tau)$ be semicompact and if possible let $\underset{G_A \in \Psi }{\overset{\sim}{\bigcap}} ssclG_A = \Phi_A$ for some family $\Psi$ of soft sets with the FIP. So, $\underset{G_A \in \Psi }{\overset{\sim}{\bigcup}} (ssclG_A)^c = U_A \Rightarrow \Upsilon =\{(ssclG_A)^c~|~ G_A \in \Psi \}$ is a semiopen cover of $U_A$. Then by semicompactness of $U_A$, $\exists$ a finite subcover $\omega$ of $\Upsilon $. i.e., $\underset{G_A \in \omega }{\overset{\sim}{\bigcup}} (ssclG_A)^c = U_A \Rightarrow \underset{G_A \in \omega }{\overset{\sim}{\bigcup}} G_A^c = U_A \Rightarrow \underset{G_A \in \omega }{\overset{\sim}{\bigcap}} G_A = \Phi_A $, a contradiction. Hence $\underset{G_A \in \Psi }{\overset{\sim}{\bigcap}} ssclG_A \neq \Phi_A$.

Conversely, we have $\underset{G_A \in \Psi }{\overset{\sim}{\bigcap}} ssclG_A \neq \Phi_A$,  for every family $\Psi$ of soft sets with FIP. Assume $(U_A, \tau)$ is not semicompact. Then $\exists$ a family $\Upsilon $ of semiopen soft sets covering $U$ without a finite subcover. So for every finite subfamily $\omega $ of $\Upsilon $ we have
$\underset{G_A \in \omega }{\overset{\sim}{\bigcup}} G_A \neq U_A \Rightarrow \underset{G_A \in \omega }{\overset{\sim}{\bigcap}} G_A^c \neq \Phi_A \Rightarrow \{G_A^c~|~ G_A \in \Upsilon\}$ is a family of soft sets with FIP. Now $\overset{\sim}{\underset{G_A\in \Upsilon}{\bigcup}}G_A = U_A \Rightarrow \overset{\sim}{\underset{G_A\in \Upsilon}{\bigcap}}G_A^c = \Phi_A \Rightarrow \overset{\sim}{\underset{G_A\in \Upsilon}{\bigcap}}sscl(G_A^c) = \Phi_A $, a contradiction. 
\end{proof}

\begin{thm}
Semicontinuous image of a soft semicompact space is soft compact.
\end{thm}
\begin{proof}
Let $f: U_A \rightarrow U_B$ be a semicontinuous function from a semicompact soft topological space $(U_A, \tau)$ to $(U_B, \delta)$. Take a soft open cover $\{(G_B)_\lambda~|~\lambda \in \Lambda\}$ of $U_B \Rightarrow \{f^{-1}((G_B)_\lambda)~|~\lambda \in \Lambda\}$ forms a soft semiopen cover of $U_A \Rightarrow \exists$ a finite subset $\Delta$ of $\Lambda$ such that $\{f^{-1}((G_B)_\lambda)~|~\lambda \in \Delta\}$ forms a semiopen cover of  $U_A \Rightarrow \{(G_B)_\lambda)~|~\lambda \in \Delta\}$ forms a finite soft opencover of $U_B$.
\end{proof}

\begin{thm}
Semiclosed subspace of a semicompact soft topological space is soft semicompact.
\end{thm}
\begin{proof}
Let $U_B$ a semiclosed subspace of a semicompact soft topological space $(U_A, \tau)$ and $\{(G_B)_\lambda~|~\lambda \in \Lambda\}$ be a semiopen cover of $U_B \Rightarrow $ for each $(G_B)_\lambda, \exists$ a semiopen soft set $G_A$ of $U_A$ such that $G_B = G_A \overset{\sim}{\bigcap} U_B$. Then the family 
$\{(G_B)_\lambda~|~\lambda \in \Lambda\} \overset{\sim}{\bigcup} (U_A -U_B)$ is a soft semi open cover of $U_A$, which has a finite subcover. So $\{(G_B)_\lambda~|~\lambda \in \Lambda\}$ has a finite subfamily to cover $U_B$. Hence $U_B$ is semicompact.
 
\end{proof}

\section{Semi Connectedness in soft topological spaces}

In this section, we introduce and study the notion of semi connectedness in a soft topological space.

\begin{defn}
 Two soft sets $F_A$ and $G_B$ are said to be disjoint if $A \cap B = \phi$ and $F(a) \cap G(b) = \phi, \forall a\in A, b\in B$.
\end{defn}

\begin{defn}
A soft semiseparation of soft topological space $(U_A, \tau)$ is a pair $F_A, G_A$ of disjoint nonnull semiopen sets
 whose union is $U_A$.

If there doesn't exist a  soft semiseparation of $U_A$, then the soft topological space is said to be soft semiconnected, otherwise soft semidisconnected
\end{defn}

\begin{thm}
 If the soft sets $H_A$ and $G_A$ form a soft semiseparation of $U_A$, and if $V_B, B \subset A$ is a soft semiconnected subspace of $U_A$,
 then $V_B \overset{\sim}{\subset} H_A$ or $V_B \overset{\sim}{\subset} G_A$.
\end{thm}

\begin{proof}
 Since $H_A$ and $G_A$ are disjoint semiopen soft sets, so are $H_A \overset{\sim}{\bigcap} V_B$ and $G_A \overset{\sim}{\bigcap} V_B$ and their soft union gives $V_B$, i.e. they would constitute a soft semiseparation of $V_B$, a contradiction.
 Hence, one of $H_A \overset{\sim}{\bigcap} V_B$ and $G_A \overset{\sim}{\bigcap} V_B$ is empty and so $V_B$ is entirely contained in on of them.
\end{proof}

\begin{thm}
 Let $V_A$ be a soft semiconnected subspace of $U_A$. If $V_A \overset{\sim}{\subset} K_A \overset{\sim}{\subset} \overline{V_A}$, then $K_A$ is also soft semiconnected. 
\end{thm}

\begin{proof}
Let the soft set $K_A$ satisfies the hypothesis.  If possible, let $F_A$ and $G_A$ form a soft semiseparation of $K_A$.
Then by theorem 5.3, $V_A \overset{\sim}{\subset} F_A$ or $V_A \overset{\sim}{\subset} G_A$. Let $V_A \overset{\sim}{\subset} F_A 
\Rightarrow sscl(V_A) \overset{\sim}{\subset} ssclF_A$; since $ssclF_A$ and $G_A$ are disjoint, $V_A$ cannot intersect $G_A$.
This contradicts the fact that $G_A$ is a nonempty subset of $V_A$.  
\end{proof}

\begin{thm}
A soft topological space $(U_A, \tau)$ is semidisconnected $\Leftrightarrow \exists$ a nonnull proper soft subset of $U_A$ which is both semiopen and semiclosed.
\end{thm}
 
\begin{proof}
Let $K_A$ be a nonnull proper soft subset of $U_A$ which is both semiopen and semiclosed. Now $H_A=(K_A)^c$ is nonnull proper subset of $U_A$ which is also both semiopen and semiclosed $\Rightarrow ssclK_A= K_A$ and $ssclH_A = H_A \Rightarrow U_A$ can be expressed as the soft union of two semiseparated soft sets $K_A, H_A$ and so is semidisconnected.

Conversely, let $U_A$ be semidisconnected$\Rightarrow \exists$ nonnull soft subsets $K_A$ and $H_A$ such that 
$ssclK_A \overset{\sim}{\bigcap} H_A = \Phi, K_A \overset{\sim}{\bigcap} ssclH_A = \Phi$ and 
$K_A \overset{\sim}{\bigcup} H_A =U_A$. Now $K_A \overset{\sim}{\subseteq} ssclK_A$ and 
$ssclK_A \overset{\sim}{\bigcap} H_A = \Phi_A \Rightarrow K_A \overset{\sim}{\bigcap} H_A = \Phi_A \Rightarrow H_A = (K_A)^c$.
Then $K_A \overset{\sim}{\bigcup} ssclH_A = U_A$ and $K_A \overset{\sim}{\bigcap} ssclH_A = \Phi_A \Rightarrow K_A = (ssclH_A)^c$ and similarly $H_A = (ssclK_A)^c \Rightarrow K_A, H_A$ are semiopen sets being the complements of semiclosed soft sets. Also $H_A = (K_A)^c \Rightarrow$ they are also semiclosed.  
 
\end{proof}

\begin{thm}
Semicontinuous image of a soft semiconnected soft topological space is soft connected.
\end{thm}

\begin{proof}
Let $f: U_A \rightarrow V_B$ be a semicontinuous function from a semiconnected soft topological space $(U_A, \tau)$ to a soft topological space$(V_B, \delta)$. It suffices to consider the surjective function  $g: U_A \rightarrow f(U_A)$.
Suppose $f(U_A) = K_B \overset{\sim}{\bigcup} H_B$ be a soft separation. i.e., $K_B$ and $H_B$ are disjoint soft open sets whose union is $f(U_A) \Rightarrow f^{-1}(K_B)$ and $f^{-1}(H_B)$ are disjoint soft semiopen sets whose union is $U_A$. So, $f^{-1}(K_B)$ and $f^{-1}(H_B)$ form a soft semiseparation of $U_A$, a contradiction.
\end{proof}

\begin{thm}
Irresolute image of a soft semiconnected soft topological space is soft semiconnected.
\end{thm}

\section{Semi Separation Axioms}
 
Here we consider different types of separation axioms for a soft topological space using semiopen and semiclosed soft sets.

\begin{defn}
 A soft topological space $(U_A, \tau)$ is said to be a soft semi $T_0-$ space if for two disjoint soft points $e_G, e_F, \exists$
 a semiopen set containing one but not the other.
\end{defn}
\begin{exm}
A discrete soft topological space is a soft semi $T_0-$ space since every $e_F \in U$ is a semiopen soft set in the discrete space.
\end{exm}

\begin{thm}
 A soft subspace of a soft semi $T_0-$ space is soft semi $T_0$.
\end{thm}

\begin{proof}
Let $V_B$ be a soft subspace of a soft semi $T_0-$ space $U_A$ and let $e_F, e_G$ be two distinct soft points of $V_B$. Then these soft points are also in $U_A \Rightarrow \exists$ a semiopen soft set $H_A$ containing one soft point but not the other$\Rightarrow H_A \overset{\sim}{\bigcap} V_B$ is a semiopen soft set containing one soft point but not the other.
\end{proof}

\begin{defn}
 A soft topological space $(U_A, \tau)$ is said to be a soft semi $T_1-$ space if for two distinct soft points $e_F,  e_G \in U_A, \exists$ soft semiopen sets $H_A$ and $G_A$ such that \\
 $e_F \overset{\sim}{\in} H_A$ and $e_G \overset{\sim}{\notin} H_A$;\\
 $e_G \overset{\sim}{\in} G_A$ and $e_F \overset{\sim}{\notin} G_A$.
\end{defn}

\begin{thm}
If every soft point of a soft topological space $(U_A, \tau)$ is a semiclosed soft set then $(U_A, \tau)$ is a soft semi $T_1-$ space.
\end{thm}

\begin{proof}
Let $e_F$ be a soft point of $U_A$ which is a semiclosed soft set $\Rightarrow (e_F)^c$ is a semiopen soft set. Then for distinct soft points $e_F, e_G$, we have  $(e_F)^c, (e_G)^c$ are semiopen soft sets such that 
$e_G \overset{\sim}{\in} (e_F)^c$ and $e_G \overset{\sim}{\notin} e_F$;
$e_F \overset{\sim}{\in} (e_G)^c$ and $e_F \overset{\sim}{\notin} e_G$.
\end{proof}
\begin{thm}
 A soft subspace of a soft semi $T_1-$ space is soft semi $T_1$.
\end{thm}
\begin{defn}
 A soft topological space $(U_A, \tau)$ is said to be a soft semi $T_2-$ space if and only if for distinct soft points $e_F, e_G \in U_A, \exists$ disjoint soft semiopen sets $H_A$ and $G_A$ such that $e_F \overset{\sim}{\in} H_A$ and $e_G \overset{\sim}{\in} G_A$.
\end{defn}

\begin{thm}
A soft subspace of a soft semi $T_2-$ space is soft semi $T_2$.
\end{thm}

\begin{proof}
Let $(U_A, \tau)$ be a soft semi $T_2-$ space and $V_B$ be a soft subspace of $U_A$, where $B \subset A $ and $V \subset  U$. Let $e_F$ and $e_G$ be two distinct soft points of $U_B$. $U_A$ is soft semi $T_2 \Rightarrow \exists$ two disjoint soft semiopen sets $H_A$ and $G_A$
 such that $e_F \overset{\sim}{\in} H_A$, $e_G \overset{\sim}{\in} G_A$. Then 
$H_A \overset{\sim}{\bigcap} U_B$ and $G_A \overset{\sim}{\bigcap} U_B$ are semiopen softs sets satisfying the requirements for $U_B$ to be a soft semi $T_2-$ space.
\end{proof}

\begin{defn}
 A soft topological space $(U_A, \tau)$ is said to be a soft semiregular space if for every soft point $e_K$  and semiclosed soft set $F_A$ not containing $e_K, \exists$ disjoint soft semiopen sets $G_{A_1}, G_{A_2}$ such that $e_K \overset{\sim}{\in} G_{A_1}$ and $F_A \overset{\sim}{\subseteq} G_{A_2}$.

 A soft semiregular $T_1-$ space is called a soft semi $T_3-$ space, 
\end{defn}
\begin{rem}
It can be shown that the property of being soft semi $T_3$ is hereditary. 
\end{rem}

\begin{rem}
Every soft semi $T_3-$ space is soft semi $T_2-$ space, every soft semi $T_2-$ space is soft semi $T_1-$ space and every soft semi $T_1-$ space is soft semi $T_0-$ space. 
\end{rem}

\begin{thm} 
Every soft semicompact semi $T_2-$ space is soft semi $T_3$.
\end{thm} 
 
\begin{proof}
It suffices to show every semicompact soft topological space is semiregular. Let $e_L$ be a soft point and $F_A$ be a semiclosed soft set  not containing the point $\Rightarrow e_L \overset{\sim}{\in} (F_A)^c$. Now for each soft point $e_F, \exists$ disjoint semiopen soft sets $(K_A)_1$ and $(H_A)_1$ such that
$e_F \overset{\sim}{\in} (K_A)_1$ and $e_L \overset{\sim}{\in} (H_A)_1$\\
So the collection $\{(K_A)_\lambda ~|~ \lambda \in \Lambda\}$ forms a semiopen cover of $F_A$. Now $F_A$ is a semiclosed soft set $\Rightarrow F_A$ is semicompact. Hence $\exists$ a finite subfamily $\Delta$ of $\Lambda$ such that $F_A \overset{\sim}{\subseteq} \overset{\sim}{\underset{\lambda \in \Delta}{\bigcup}} (K_A)_\lambda$.

Take $H_A = \overset{n}{\underset{i=1}{\bigcap}} (H_A)_i$ and $K_A = \overset{n}{\underset{i=1}{\bigcup}} (K_A)_i$. Then $H_A, K_A$ are disjoint semiopen sets such that $e_L$ is a soft point of $H_A$ and 
$F_A \overset{\sim}{\subseteq} K_A$.    
\end{proof}

\begin{defn}
 A soft topological space $(U_A, \tau)$ is said to be a soft seminormal space if for every pair of disjoint semiclosed soft sets $F_A$ and $K_A, \exists$ two disjoint soft semiopen sets $H_{A_1}, H_{A_2}$ such that 
$F_A \overset{\sim}{\subseteq} H_{A_1}$ and $K_A \overset{\sim}{\subseteq} H_{A_2}$.

A soft seminormal $T_1-$ space is called a soft semi $T_4-$ space. 
\end{defn}

\begin{rem}
Every soft semi $T_4-$ space is soft semi $T_3$.
\end{rem}
\begin{thm}
A soft topological space $(U_A, \tau)$ is seminormal iff for any semiclosed soft set $F_A$ and semiopen soft set $G_A$ containing $F_A$, there exists an semiopen soft set $H_A$ such that $F_A \overset{\sim}{\subset} H_A$ and $sscl(H_A) \overset{\sim}{\subset} G_A$. 
\end{thm}
 \begin{proof}
Let $(U_A, \tau)$ be seminormal space and $F_A$ be a semiclosed soft set and $G_A$ be a semiopen soft set containing $F_A \Rightarrow F_A $ and $(G_A)^c$ are disjoint semiclosed soft sets $\Rightarrow \exists$ two disjoint semiopen soft sets $H_{A_1}, H_{A_2}$ such that 
$F_A \overset{\sim}{\subset} H_{A_1}$ and $(G_A)^c \overset{\sim}{\subseteq} H_{A_2}$. Now 
$H_{A_1} \overset{\sim}{\subset} (H_{A_2})^c \Rightarrow ssclH_{A_1} \overset{\sim}{\subset} sscl(H_{A_2})^c =(H_{A_2})^c$ 
Also, $(G_A)^c \overset{\sim}{\subset} H_{A_2} \Rightarrow (H_{A_2})^c \overset{\sim}{\subset} G_A \Rightarrow ssclH_{A_1} \overset{\sim}{\subset}(G_A)$.

Conversely, let $L_A$ and $K_A$ be any disjoint pair semiclosed soft sets $\Rightarrow L_A \overset{\sim}{\subset} (K_A)^c $, then by hypothesis there exists an semiopen soft set $H_A$ such that $L_A \overset{\sim}{\subset} H_A$ and $ssclH_A \overset{\sim}{\subset} (K_A)^c \Rightarrow (K_A)\overset{\sim}{\subset} (ssclH_A)^c \Rightarrow (H_A)$ and $(ssclH_A)^c$ are disjoint semiopen soft sets such that $L_A \overset{\sim}{\subset} H_A$ and $K_A \overset{\sim}{\subset} (ssclH_A)^c$.
\end{proof}

\begin{thm}
Let $f: (U_A, \tau) \rightarrow (U_B, \delta)$ be a soft surjective function which is both irresolute and soft semiopen. If $U_A$ is soft seminormal space then so is $U_B$.
\end{thm}

\begin{proof}
 Take a disjoint pair $L_A$ and $M_A$ of semiclosed soft sets of $U_B \Rightarrow f^{-1}(L_A)$ and $f^{-1}(M_A)$ are disjoint semiclosed soft sets of $U_A \Rightarrow \exists$ disjoint semiopen soft sets $G_A$ and $H_A$ such that $f^{-1}(L_A) \overset{\sim}{\subset} G_A$ and $f^{-1}(M_A) \overset{\sim}{\subset} H_A \Rightarrow L_A \overset{\sim}{\subset} f(G_A)$ and $M_A \overset{\sim}{\subset} f(H_A) \Rightarrow f(G_A)$ and $f(H_A)$ are disjoint open soft sets of $U_B$.
\end{proof}

\begin{thm}
A semiclosed soft subspace of a soft seminormal space is soft seminormal.
\end{thm} 
 
\begin{proof}
Let $V_B$ be a semiclosed soft subspace of a soft seminormal space $U_A$. Take a disjoint pair $L_B$ and $M_B$ of semiclosed sets of $V_B \Rightarrow \exists$ disjoint semiclosed soft sets $L_A$ and $M_A$ such that $L_B = L_A \overset{\sim}{\bigcap}V_B,  M_B = M_A \overset{\sim}{\bigcap}V_B$. Now by soft seminormality of $U_A,~ \exists$ disjoint semiopen soft sets $G_A$ and $H_A$ such that $L_A \overset{\sim}{\subset} G_A$ and $M_A \overset{\sim}{\subset} H_A \Rightarrow L_B \overset{\sim}{\subset} G_A \overset{\sim}{\bigcap}V_B $ and $M_B \overset{\sim}{\subset} H_A \overset{\sim}{\bigcap}V_B$
\end{proof}

\begin{thm}
Every soft semicompact semi $T_2-$ space is seminormal.  
\end{thm}

\begin{proof}
Let $(U_A, \tau)$ be a semicompact semi $T_2-$ space. Take a disjoint pair $L_A$ and $M_A$ of semiclosed sets. By theorem 6.12, for each soft point $e_L, ~\exists$ disjoint semiopen soft sets $G_{e_{\scriptscriptstyle L}}$ and $H_{e_{\scriptscriptstyle L}}$ such that  $e_L \overset {\sim}{\subset} G_{e_{\scriptscriptstyle L}} $ and $M_A \overset {\sim}{\subset} H_{e_{\scriptscriptstyle L}}$. So the collection $\{G_{e_{\scriptscriptstyle Li}}~|~ e_{\scriptscriptstyle L} \overset{\sim}{\subset} G_{e_{\scriptscriptstyle L}}, i \in \Lambda \}$ is a semiopen cover of $G_{e_{\scriptscriptstyle L}}$. Then by theorem 4.7, $\exists$ a finite subfamily $\{G_{e_{\scriptscriptstyle Li}}~|~ i=1, 2, .~.~.~, n \}$ such that $G_{e_{\scriptscriptstyle L}} \overset{\sim}{\subset} \overset{\sim}{\bigcup} \{G_{e_{\scriptscriptstyle Li}}~|~ i=1, 2, .~.~.~, n \}$. Take $G_A = \overset{\sim}{\bigcap} \{G_{e_{\scriptscriptstyle Li}}~|~ i=1, 2, .~.~.~, n \}$ and $H_A = \overset{\sim}{\bigcap} \{H_{e_{\scriptscriptstyle Li}}~|~ i=1, 2, .~.~.~, n \}$. Then $G_A$ and $H_A$ are disjoint semiopen soft sets such that $L_A \overset{\sim}{\subset} G_A$ and $M_A \overset{\sim}{\subset} H_A$. Hence $U_A$ is seminormal.
\end{proof}

\end{document}